\documentclass{amsart}
\usepackage{amssymb,amsfonts,amsmath}
\usepackage[all]{xy}

\newtheorem{theorem}{Theorem}[section]

\newtheorem{definition}{Definition}[section]
\newtheorem{claim}{Claim}[section]
\newtheorem{lemma}{Lemma}[section]

\theoremstyle{remark}
\newtheorem{remark}{Remark}
\newcommand{\g}{\gamma}
\newcommand{\si}{\sigma}
\newcommand{\ti}{\tilde}
\newcommand{\om}{\omega}

\newcommand{\fr}{\frac}
\newcommand{\la}{\lambda}
\newcommand{\lf}{\left}
\newcommand{\rg}{\right}
\newcommand{\ep}{\epsilon}
\newcommand{\al}{\alpha}
\newcommand{\real}{\mathbb{R}}

\newcommand{\inte}{\mathrm{int}}

\newcommand{\be}{\beta}
\newcommand{\de}{\delta}

\newcommand{\bl}{\bigl}
\newcommand{\br}{\bigr}
\newcommand{\m}{\mathcal}
\newcommand{\p}{\partial}
\newcommand{\wi}{\widetilde}
\newcommand{\va}{\varphi}
\newcommand{\ta}{\tau}

\begin{document}

\title[Complete negatively curved immersed ends in $\mathbb{R}^3$] {Complete negatively curved immersed ends in $\mathbb{R}^3$}
\author{S\'ergio Mendon\c ca}
\address{Departamento de An\'alise, Instituto de Matem\'atica,
Universidade Federal Fluminense, Niter\'oi, RJ, CEP 24020-140,
Brasil} \email{sergiomendonca@id.uff.br}

\dedicatory {To my beloved wife, Cristina Marques, whose love and sincerity inspire me}


\begin{abstract} This paper extends, in a sharp way, the famous Efimov's Theorem to 
immersed ends in $\real^3$. More precisely, let $M$ be a non-compact connected surface with compact boundary. Then there is 
no complete isometric immersion of $M$ into $\Bbb R^3$ satisfying that $\int_M |K|=+\infty$ and $K\le-\kappa<0$, where  $\kappa$ is a positive constant and $K$ is the Gaussian curvature of 
$M$. In particular Efimov's Theorem holds for complete Hadamard immersed surfaces, whose Gaussian curvature $K$ is bounded away from zero outside a compact set.  
 \end{abstract}

\maketitle

In 1901 Hilbert ([Hi]) proved that there is no complete immersed hyperbolic plane $\Bbb H^2$ in $\Bbb R^3$. In
1902 Holmgren ([Ho]) presented a new version with a more rigorous proof. Blaschke ([Bs]) and Bieberbach ([Bi])
presented new versions of the proof. In 1955 Blan\v usa (\cite{bl}) presented an example 
of a smooth complete 
isometric embedding 
of $\Bbb H^2$ in $\Bbb R^6$. In 1960 Rozendorn (\cite{ro}) obtained a 
smooth complete isometric immersion of $\Bbb H^2$ in $\Bbb R^5$. 

 In 1936
Cohn-Vossen (\cite{cv}) have conjectured that the hyperbolic plane in Hilbert's Theorem could be replaced by a complete immersed
surface with Gaussian curvature not greater than a negative constant. The solution for this problem came only in 1964 with the work of
Efimov (\cite{e}):

\begin{theorem}[Efimov's Theorem] There is no complete isometric immersion $\varphi:M\to \Bbb R^3$ with Gaussian curvature
$K\le -\kappa<0$, where $M$ is a connected surface.
\end{theorem}

An extension of the Efimov's Theorem to higher dimensions was the
work of B. Smith and F. Xavier in \cite{sx}, in which it is proved that there exists no
codimension one complete isometric immersion $f:M\to \Bbb R^n$ with the Ricci curvature ${\rm Ric}_M\le -\kappa<0$,  provided that $n=3$, or that $n\ge 4$ and
the sectional curvatures of $M$ do not assume all values in $\Bbb R$. 

Tilla Klotz Milnor in [Mi] published a more detailed version of the proof of the Efimov's Theorem.
In her version the immersion $\varphi:M\to \Bbb R^3$ is $C^2$ and the
induced Riemannian metric on $M$ is just supposed to be $C^1$.
With this hypothesis we may not define the Gaussian curvature by the usual intrinsic method.
Instead we follow [Mi] and define
$K = \frac{eg-f^2}{EG -F^2}$ for the Gaussian curvature in terms of the first and second fundamental forms of the immersion: $I = E dx^2 +2F dxdy + G dy^2$ and $II=e\, dx^2+2 f\,  dx dy+g\, dy^2$.

Our main result is the following

\begin{theoremA} {\it Let $M$ be a non-compact connected surface with compact boundary.
Then
 does not exist any $C^2$ immersion $\varphi:M\to \Bbb R^3$
 inducing a $C^1$ complete Riemannian metric on $M-\p M$ with the Gaussian curvature
 satisfying
 $\int_{M} |K|=+\infty$ and $K\le-\kappa<0$ for
some constant $\kappa$.}
\end{theoremA}

Note that if $\p M=\emptyset$ Theorem A reduces to the Efimov's Theorem. Indeed, 
if there exists an immersion $\varphi:M\to \Bbb R^3$ as in the statement of the Efimov's Theorem, we may compose it with the universal covering $P:\ti M\to M$. The induced 
metric on $\ti M$ will have infinite area (see \cite{mi}), hence the fact that $K\le -\kappa<0$ implies that the total curvature of $\ti M$ is infinite, which contradicts Theorem A.

By applying 
Theorem A to the complement of an open ball, we obtain the following application to Hadamard 
surfaces.
 
\begin{corollaryB} Let $M$ be an open simply-connected surface without boundary. Then
 does not exist any $C^2$ immersion $\varphi:M\to \Bbb R^3$
 inducing a $C^1$ complete metric on $M$ with the Gaussian curvature
 satisfying  $K\le 0$ on $M$ and $K\le-\kappa<0$ on the complement $M-B$, where
 $B$ is an open ball and $\kappa$ is a positive constant.
\end{corollaryB}

\begin{remark} Each hypothesis in Theorem A is essential. Indeed, consider the 
set 
$$S=\lf\{\lf(\cos u\sin t,\sin u\sin t,\cos t+\log \tan \fr t2\rg)\bigm|u\in 
\lf[0,2\pi\rg], t\in\lf[\fr \pi 2+\ep\,,\,\pi\rg)\rg\},$$
for some $\ep >0$, which is a smooth surface with boundary, contained 
in a pseudosphere with Gaussian curvature $-1$ and finite area. Thus 
the surface $S\subset \Bbb R^3$ 
shows that 
the condition that the total curvature is infinite may not be dropped. 
Its universal covering $P:M=\ti S\to S\subset\Bbb R^3$ with the induced metric shows that the compactness of the boundary $\p M$ is also essential, since the other conditions hold. 
Now consider the incomplete surface  $S'=S-\p S$ and its universal covering $P':H\to S'\subset \Bbb R^3$ with induced 
metric. It shows that the hypothesis that $\varphi$ is complete may not be removed from Theorem A. The helicoid shows that 
the condition $K\le-\kappa<0$ may not be replaced by the condition $K<0$.  M. Kuiper showed in \cite{ku} that the condition that $\varphi$ is $C^2$ is essential even in the Hilbert's theorem.  
\end{remark}

\begin{remark} As a byproduct of this paper we provide a more detailed presentation of the proof of the Efimov's theorem. 
\end{remark}

To prove Theorem A, the proof of the Efimov's Theorem will be used  as written
in [Mi]. The author would like to thank Heudson Mirandola, Cristina Marques and
Manolo Heredia for useful discussions during the reading of  that paper.

\begin{remark} It came to our knowledge that the paper [GMT] proves 
Theorem A independently.
\end{remark}

\section{\bf Notations}
Given a Riemannian metric $\om$ on a manifold $M$, let $A_\om$, $L_\om$, $d_\om$ denote, respectively, the area, length, distance associated  to the metric $\om$. Similarly a $\om$-geodesic will
denote a geodesic with respect to the metric $\om$. Given a subset $C\subset M$ we set
$$B_\om(C,r)=\{x\in M\bigm|d_\om(x,C)<r\},$$
$$\bar B_\om(C,r)=\{x\in M\bigm|d_\om(x,C)\le r\}$$
and
$$S_\om(C,r)=\{x\in M\bigm|d_\om(x,C)=r\}.$$

\section{\bf An idea of the proof of Theorem A}

To give an idea of the proof of Theorem A, we consider an 
immersion $\varphi:M\to \Bbb R^3$ as in the statement of 
Theorem A. By passing to the orientable double cover with induced metric if necessary,
we still have the same curvature conditions, the completeness of the 
induced metric $\al$ and the compactness of the boundary. 
Thus we may assume, without loss of generality, that $M$ is orientable. As a consequence  
 there exists a $C^1$ Gauss map $N:M\to S^2$, where $S^2$ is the sphere with  
the standard round metric $\nu$ of curvature $1$. Since $K<0$ the map $N$ is  a local $C^1$ diffeomorphism on $M-\p M$. 
Let $\be$ be the $C^0$ Riemannian metric induced by $N$ on $M-\p M$. We have:
\begin{equation}\label{infinite_area} A_\be(M)=\int_MdA_\be=\int_M|{\rm det}(dN)|\,dA_\al=\int_M|K|\,dA_\al=+\infty.
\end{equation}
In Lemma \ref{Misbounded} below we will show that $(M,\be)$ is
bounded. To do this we will fix some point $p\in M$ far
from the boundary $\p M$. Then we will apply on
a large ball centered at $p$ similar ideas as in [Mi] obtaining a contradiction if 
the distance from $\p M$ is greater than $5\pi$. 
We will present some arguments in a way different from [Mi]. In some points 
our proof is shorter and in another points  we prefer to present more details, 
in order to make the proof
clearer. We also need to be careful to assure that the constructions used never involve points in $\p M$. Finally we will show that boundedness in this case implies pre-compactness 
and finite area of $(M,\be)$, contradicting equation (\ref{infinite_area}) and proving Theorem A.

\section{\bf Some basic facts}

The following simple lemma, which will be used in the proof of Lemma \ref{general_lemma},   
is possibly known, but we didn't find it in the literature.
For completeness we will present its proof in the Appendix, which is partially inspired in the proof of Lemma 3.1 in [Mi]. 

\begin{lemma} \label{topologies} Let $(S,g)$ be a Riemannian smooth surface. Let $D\subset S$ be a connected surface with piecewise smooth boundary with internal angles at the vertices
different from $0$ and $2\pi$.
For $p,q\in D$, let $d_g(p,q)$ denote the distance induced by
the Riemannian metric $g$, and $d_{\inte}(p,q)$ the infimum of the $g$-lengths of
piecewise smooth curves $\g:[0,1]\to
D$ joining $p$ to $q$, such that $\g\bl((0,1)\br)\subset \inte(D)$, where $\inte(D)$ is the set of
interior points of $D$ with respect to the metric $d_g$. Then $d_\inte$ is a distance on $D$, and the distances
$d_g$ and $d_\inte$ induce the same topology on $D$.
\end{lemma}

Take positive numbers $r,s$ with $s+2r<\pi$. Consider a unit speed minimal $\nu$-geodesic $\g:\Bbb R\to S^2$. Fix $z=\g(0)$ and $u=\g(s)$. Consider the antipodal 
points $p=\g\lf(\fr{s-\pi}2\rg)$,  $p^*=\g\lf(\fr{s+\pi}2\rg)$, which satisfy 
$d_\nu(p,z)=d_\nu(p^*,u)=\fr{\pi-s}2$. Set
\begin{equation}\label{Xdef}X=\bar B_\nu(z,r)\,\cup\,\g([0,s])\,\cup\,\bar B_\nu(u,r).
\end{equation}

We will recall the construction of the convex hull of $X$, which will be needed 
in the proof of Lemma \ref{small_delta} below. 
Fixing the antipodal points $p$ and $p^*$, we rotate $\g$ in both directions until we
obtain exactly two geodesics $\g_1:\lf[\fr{s-\pi} 2,\fr{s+\pi} 2\rg]\to S^2$ and $\g_2:
\lf[\fr{s-\pi} 2,\fr{s+\pi} 2\rg]\to S^2$ from $p$ to $p^*$, which intersect tangentially both $S_\nu(z,r)$ and $S_\nu(u,r)$. This is
possible since $s+2r<\pi$. For each $i\in\{1,2\}$, consider points $p_i$ and $q_i$ given by 
$$\{p_i\}=S_\nu(z,r)\,\cap\, \g_i\lf(\lf[\fr{s-\pi} 2,\fr{s+\pi} 2\rg]\rg), \,\{q_i\}=S_\nu(u,r)\,\cap\, \g_i\lf(\lf[\fr{s-\pi} 2,\fr{s+\pi} 2\rg]\rg).$$ 
Let $Y(X)$ be the compact domain containing $X$, whose boundary 
of $C^1$ class is the image of a curve which follows $\g_1$ from 
$p_1$ to $q_1$, then the arc in $S_\nu(u,r)$ from $q_1$ to $q_2$ which contains 
$\g(s+r)$, then the image of 
$\g_2$ in the opposite direction from $q_2$ to $p_2$, and then the arc in $S_\nu(z,r)$ from $p_2$ to $p_1$ which 
contains $\g(-r)$. Since $r<\fr\pi 2$, the balls $B_\nu(z,r)$ and $B_\nu(u,r)$ are 
strongly convex, hence it is easy to prove the following

\begin{lemma} Given positive numbers $r,s$ such that $s+2r<\pi$ and 
$X$ defined as in {\rm(\ref{Xdef})}, the set $Y=Y(X)$ is the convex hull of $X$, and $Y$ is strongly convex. 
\end{lemma}

Fix a point $p\in S^2$ and $0<r\le\fr\pi 2$. Consider the circle $\m S=S_\nu(p,r)=S_\nu(p^*,\pi-r)$. Let $W_{pr}$ be the set of pairs $(x,v)$ in the normal fiber bundle of $\m S$ satisfying  one of the following three conditions: $v=0$; $0<|v|<r$ and $v$ points to $B_\nu(p,r)$;  $0<|v|<\pi-r$ and $v$ points to 
$B_\nu(p^*,\pi-r)$. By using spherical coordinates it is easy to show that the normal  exponential map $\exp^\perp:W_{pr}\to S^2-\{p,p^*\}$ is a diffeomorphism and 
 $\exp^\perp(\p W_{pr})=\{p,p^*\}$. Thus it is easy to prove the following well known 
Lemma, which will be used in the proof of Lemma \ref{small_delta} below.

\begin{lemma}\label{sphere_radius_r} Fix $0<r\le\fr\pi 2$ and $p\in S^2$. 
If $x\in B_\nu(p,r)-\{p\}$ 
and $\g:[0,d]\to S^2$ is a unit speed geodesic from $x$ to $\m S=S_\nu(p,r)$ with 
$0<d<r$ and $\g'(d)$ orthogonal to $\m S$, then $\g$ is the unique unit speed 
geodesic from $x$ to $\m S$ such that $d_\nu(x,\m S)=L_\nu(\g)$. If 
$z\in B_\nu(p^*,\pi-r)-\{p^*\}$ 
and $\si:[0,e]\to S^2$ is a unit speed geodesic from $z$ to $\m S$ with 
$0<e<\pi-r$ and $\g'(e)$ orthogonal to $\m S$, then $\si$ is the unique unit speed 
geodesic from $z$ to $\m S$ such that $d_\nu(z,\m S)=L_\nu(\si)$.
\end{lemma}

\begin{lemma} \label{Ls} Consider unit speed $\nu$-geodesics $\g:[0,\mu]\to S^2$, with 
$0<\mu<\pi$ and
$\eta:[0,\pi]\to S^2$ with $\eta(0)=\g\lf(\fr \mu 2\rg)$ and $\eta'(0)$ orthogonal to $\g$. 
For $0\le s<\fr\pi 2$, set $q_{s}=
\eta\lf(-\fr\pi 2+s\rg)$ and $q^s
=\eta\lf(\fr\pi 2-s\rg)$. If $z=\g(0)$ and $u=\g(\mu)$, consider the distance 
$d_s=d_\nu(q_s,z)=d_\nu(q_s,u)=d_\nu(q^s,z)=d_\nu(q^s,u)$. Set $D_0=\g\bl([0,\mu]\br)$ 
and 
$$D_s=\bar B_\nu(q_s,d_s)\cap \bar B_\nu(q^s,d_s),$$
if $0<s<\fr\pi 2$. Then for $0\le s<s'<\fr\pi 2$ it holds that $D_s\subset D_{s'}$.
\end{lemma}
\begin{proof} 
We first observe that 
\begin{equation} \label{pLs} \{z,u\}\subset \p L_s,
\end{equation} 
for all $0\le s<\fr\pi 2$. 

By the spherical law of cosines we have that
\begin{equation} \label{cosines}\cos d_s=\sin s\,\cos\fr{\mu}2,
\end{equation}
for $0\le s<\fr\pi 2$. 

Fix $0<s<\fr\pi 2$. Equation (\ref{cosines}) implies that $0<d_s< \fr\pi 2$, 
hence $D_s$ is strongly convex. In particular we have by (\ref{pLs}) that $D_0\subset D_s$. 
 
Now fix $0<s<s'< \fr\pi 2$.  To prove that $D_s\subset D_{s'}$ we take $x\in D_s$. Thus 
we have that $d_\nu(x,q_s)\le d_s$.  
By triangle inequality we have that
\begin{equation} \label {compare_pi}d_\nu(x,q_{s'})\le d_\nu(x,q_s)+(s'-s)\le d_s+s'<\fr\pi 2+s'<\pi.
\end{equation}

Fix $x_0\in D_s$ such that $d_\nu(x_0,q_{s'})$ is a maximum. By (\ref{compare_pi}) there exists 
a unique $\nu$-geodesic $\chi:[0,1]\to S^2$ from $q_{s'}$ to $x_0$. The 
inequality  (\ref{compare_pi}) also implies that  
$\chi$ may be extended to a 
 minimal geodesic $\ti\chi:[0,1+\ep]\to S^2$ for some small $\ep>0$. Thus the maximality of $d_\nu(x_0,q_{s'})$ implies 
 that $x_0\in\p D_s$.  

First assume that $x_0$ is in the interior of the arc  
$D_s\cap S_\nu(q_s,d_s)$.  
By the first variation formula and the maximality of 
$d_\nu(x_0,q_{s'})$, it follows from  
(\ref{compare_pi}) that $\chi'(1)$ is orthogonal to $D_s\cap S_\nu(q_s,d_s)$. 
Since $q_{s'}$ is 
distinct from $q_s$ and $(q_s)^*$, we have from Lemma \ref{sphere_radius_r} 
that the distance from 
$q_{s'}$ attains a strict minimum at $x_0$, which is a contradiction.  
Thus $x_0$ may not belong to the interior of $D_s\cap S_\nu(q_s,d_s)$.  
Since $q_{s'}$ is distinct from $q^s$ and $(q^s)^*$, we obtain similarly that 
$x_0$ may not belong to the interior of $D_s\cap S_\nu(q^s,d_s)$. We conclude 
that $x_0\in \{z,u\}$, hence $d_\nu(x,q_{s'})\le d_\nu(x_0,q_{s'})=d_{s'}$. As a consequence 
we have that $x\in \bar B_\nu(q_{s'},d_{s'})$. Similarly we show that $x\in \bar B_\nu(q^{s'},d_{s'})$, 
hence $x\in D_{s'}$. Lemma \ref{Ls} is proved. 
\end{proof}

\section{\bf Proof of Theorem A}

If $(N,\be):M\to (S^2,\nu)$ is a $C^1$ isometric immersion, by a 
$\be$-geodesic in $M$ we will mean a curve that 
locally minimizes length.  In [Mi], a metric ball $B_\be(p,r)\subset M$ was called a full geodesic disk $B_\be(p,r)\subset M$ if for any unit vector $v\in T_pM$ the geodesic starting at 
$v$ is defined on $[0,r]$. The fact that $N$ preserves length of curves implies easily 
that $N$ is injective on a full geodesic disk  $B_\be(p,r)$ if $0<r<\pi$, as well as on  
the image of any $\be$-geodesic of length smaller than $2\pi$. 
From now on, for simplicity we will use the expression \lq normal ball\rq \ instead of  \lq full geodesic disk\rq.

If $U\subset M-\p M$ is an open
set, such that $N|_U:U\to N(U)\subset S^2$ is injective and $U,N(U)$ are strongly convex 
with respect to $\be$, then $N|_U:\lf(U,d_{\be}|_{_{U\times U}}\rg)\to \lf(N(U),d_\nu|_{_{(N(U)\times N(U))}}\rg)$ is an isometry.
In particular $N:(M-\p M,\be)\to (S^2,\nu)$ is a local isometry.

From the non-compactness of $M$ and the Myers Theorem ([My]) it follows that the
metric $\be$ cannot be complete. Indeed, if $(M,\be)$ is complete and non-compact
there exist a minimal $\be$-geodesic $\g$ in $M$ with $\be$-length greater then $\pi$, which
would contradict the Myers Theorem for $N\circ\g$.

If $(M,\be)$ is some connected $C^2$ surface with compact boundary, let $\wi M$ be the metric completion of $(M,\be)$. We set $\de M=\wi M-M$. Assume 
that there exists some $C^1$ local isometry $N:(M,\be)\to (S^2,\nu)$. Since $N$ 
preserves length of curves, it is easy to see that $N$ maps Cauchy sequences in $M$ to Cauchy sequences in $S^2$, hence $N$ has a unique continuous extension
$\wi N:\wi M\to S^2$.  For some $X\subset M$ we denote by $\wi X$ the closure of $(X,d_\be|_{_{X\times X}})$ in $\wi M$. Unless otherwise stated we will consider on $X$ the topology
induced by $d_\be|_{_{X\times X}}$. Similarly, for $Y\subset S^2$ we will always consider the topology
induced by the inclusion in $S^2$.

The next definition is the natural extension of Definition 2 in [Mi] to a manifold with boundary.
\begin{definition}\label{concave} Consider a connected noncompact surface  $M$ 
and a $C^1$ local isometry $N:(M,\be)\to (S^2,\nu)$. We will say that  $\wi M$ is concave at some $q\in \de M$ if there exist $p\in S^2$ and  $\fr \pi 2<r< \pi$ such that:
\begin{enumerate}
\item $\wi N(q)\in S_\nu(p,r)$;
\item there exists $\ep>0$ such that $\lf(B_\nu\bl(\wi N(q),\ep\br)\cap B_\nu(p,r)\rg)=N(U)$, where $U$ is an  open set in $(M-\p M)$ such that $q\in \wi U$ and $\wi N$ is injective on $\wi U$.
\end{enumerate}
\end{definition}

\begin{lemma}\label{lemmaA} Let $\varphi:M\to \Bbb R^3$ be as 
in the statement of Theorem A. 
Then there is no point $q\in\de M$ 
at which $\wi M$  is concave.
\end{lemma}

The proof of Lemma \ref{lemmaA} is exactly the same as in the proof of Lemma A in [Mi],
since all arguments there refer to a neighborhood of the end point
$q$ as in Definition \ref{concave}. Lemma \ref{lemmaA} is the unique result in this paper where the completeness of 
$(M,\al)$ is needed. We observe that the hypothesis that the total curvature of 
$M$ is infinite is not used to prove Lemma \ref{lemmaA}.

The following lemma 
will be used several times in this paper. Compare condition
(\ref{partition})-(b)  in Lemma \ref{general_lemma} with Definition \ref{concave} above.

\begin{lemma}\label{general_lemma} Consider a connected surface  $M$ and a $C^1$ local isometry $N:(M,\be)\to (S^2,\nu)$ such that there exists no point $q\in\de M$ 
at which $\wi M$ is concave. Fix $w\in M-\p M$ and a compact domain $F\subset M-\p M$ containing $w$ such that $N|_F$ is injective. Assume
that there exist compact domains $\hat F_t\subset S^2$,
for $s_0\le t< T$ satisfying the following conditions:
\begin{enumerate}
\item \label{F0}$\hat F_{s_0}=N(F)$;
\item \label{monotone}if $s_0\le t\le s< T$ then $\hat F_t\subset\hat F_{s}$;
\item \label{deformation} there exists a continuous deformation $s_0\le t< T\longmapsto \si_t:S^1\to S^2$ such that
the boundary $\p \hat F_t=\si_t(S^1)$, where $\si_t$ is a piecewise smooth simple closed curve with internal angles with respect to $\hat F_{t}$ different from $0$ and $2\pi$;
\item  \label{partition} for any $s_0< t< T$ and any $z\in \p\hat F_t$, one of the following assumptions hold:
\begin{enumerate}
\item $z\in \hat F_{s_0}=N(F)$;
\item there exist $p\in S^2$ and $\fr \pi 2<r<\pi$
such that there exists a nonempty open arc $C\subset \bl(\p\hat F_t\,\cap\, S_\nu(p,r)\br)$ containing $z$, and  there exists  $\ep>0$
such that $\lf(B_\nu(z,\ep)\,\cap\, \inte\bl(\hat F_t\br)\rg)\subset
B_\nu(p,r)$;
\end{enumerate}
\item \label{star_shaped} in the case that $\p M\not=\emptyset$, it holds that,  
for each  $s_0< t< T$ and any $z\in \lf(\p\hat F_{t}\,\cap\,\bl(S^2-N(F)\br)\rg)$, there
exists a piecewise smooth curve $\g_{zt}:[0,1]\to \hat F_t$ joining $N(w)$ to $z$ with $L_\nu(\g_{zt})
<d_\be(w,\p M)$.
\end{enumerate}
Then there exists a connected set $U\subset M-\p M$ containing $F$ such that $N|_U:U\to \bigcup_{s_0\le t<T}
\hat F_t$ is a 
bijection.
\end{lemma}
\begin{proof}
Let $Z$ be the set of numbers $t\in[s_0,T)$
such that there exists a family $(F_s)_{s_0\le s\le t}$ of compact sets  contained in $M-\p M$ such that $N|_{F_{s}}:F_{s}\to \hat F_{s}$ is a bijection for any $s_0\le s\le t$ and $F=F_{s_0}\subset F_{u}\subset F_v$, if $s_0\le u\le v\le t$. By condition (\ref{F0}) 
we see easily that $s_0\in Z$. If $\sup(Z)=T$, the local 
$C^1$ diffeomorphism $N$ is injective on $U=\bigcup_{s_0\le t<T}F_t$, hence 
$N|_U:U\to N(U)$ is a homeomorphism. Since the set $N(U)
=\bigcup_{s_0\le t<T}\hat F_t$ is 
connected, we conclude that $U$ is a connected set. Thus, to prove Lemma 
\ref{general_lemma} it suffices to show that $\sup(Z)=T$. 
 
We assume by  contradiction that $t_0=\sup(Z)<T$. First we prove the following
\begin{claim} \label{t0Z}$t_0\notin Z$ and $s_0<t_0$.
\end{claim}
In fact, if $t_0\in Z$ then for each $s_0\le t\le t_0$ the map $N|_{F_{t}}:F_{t}\to \hat F_{t}$ is
injective and a local $C^1$ diffeomorphism, hence it is a homeomorphism. Furthermore it holds that $F=F_{s_0}\subset F_{u}\subset F_v\subset (M-\p M)$, if $s_0\le u\le v \le t_0$. Since
$F_{t_0}$  is compact and $N$ is a local diffeomorphism which is injective on $F_{t_0}$, it is easy to see that there exists an open set $W\subset (M-\p M)$ containing $F_{t_0}$ such that
$N|_W:W\to S^2$ is injective. By condition (\ref{deformation}) in Lemma \ref{general_lemma},
the map $t\longmapsto \si_t$ is continuous, hence the compactness of $\p \hat F_{t_0}$ 
implies that there exists $t_0<s<T$ such
that for any $t_0<t\le s$ it holds that $\hat F_t$ is contained in the open set $N(W)$, hence we may set  $F_t=
(N|_W)^{-1}(\hat F_{t})$, obtaining that
$N|_{F_{t}}:F_{t}\to \hat F_{t}$ is a homeomorphism. In particular we have that each $F_{t}$ is a compact set
and by condition (\ref{monotone}) in Lemma \ref{general_lemma} the family $(F_t)_{s_0\le t\le s}$ satisfies that $F=F_{s_0}\subset F_{u}\subset F_v\subset W\subset (M-\p M)$, if $s_0\le u\le v\le t$. We conclude that $s\in Z$, which contradicts the fact that $t_0=\sup(Z)$. Since we proved that $t_0\notin Z$ and $s_0\in Z$ we have  that $s_0<t_0$. Claim \ref{t0Z} is proved.

Set $V=\bigcup_{s_0\le t<t_0}F_t$. Since $F_u\subset F_v$ if $u\le v$, it is easy to see that $N|_V:V\to N(V)$ is injective, hence it is a homeomorphism.
We
have that
$$N(V)=N\lf(\bigcup_{s_0\le t<t_0}F_t\rg)=\bigcup_{s_0\le t<t_0}\hat F_t\subset \hat F_{t_0}.
$$
Take $x\in
\inte(\hat F_{t_0})=\hat F_{t_0}-\p\hat F_{t_0}$. By compactness we have that  $d=d_\nu(x,\p \hat F_{t_0})>0$. By compactness again and using the continuous deformation $t\longmapsto \si_t$,
we obtain that there exists $s_0<s<t_0$ such that $\hat F_{t_0}-\hat F_s\subset
B_\nu\bl(\p \hat F_{t_0},d\br)$, hence $x\in \hat F_s$.  Thus we obtain that
\begin{equation}\label{intft0}\inte\bl(\hat F_{t_0}\br)\subset N(V) \mbox{\ \  \ \ \ \ and\ \ \ \ \ }\hat F_{t_0}-N(V)\,\subset \,\hat F_{t_0}-\inte\bl(\hat F_{t_0}\br)=\p\hat F_{t_0}.
\end{equation}

\begin{claim} \label{not5a} $\wi N|_{\wi V}:\wi V\to S^2$ is injective.
\end{claim}
Fix $x,y\in\wi V$ such that $\wi N(x)=\wi N(y)$. Take sequences $x_n,y_n\in V$ such that $x_n\to x$ and $y_n\to y$ in $\wi M$.
Since $N(x_n)$ and $N(y_n)$ converge to $\wi N(x)=\wi N(y)$ in $S^2$, we have that $\wi N(x)\in \hat F_{t_0}$.
Since $N(x_n)$ and $N(y_n)$ belong to $\hat F_{t_0}$ and converge to $\wi N(x)$, Lemma \ref{topologies}
implies that $d_{\inte}(N(x_n),N(y_n))\to 0$, hence there exists a sequence of piecewise smooth
curves $\ti \g_n:[0,1]\to \hat F_{t_0}$ joining $N(x_n)$ to $N(y_n)$, with $L_\nu(\ti\g_n)\to 0$ and $\ti\g_n((0,1))\subset \inte(\hat F_{t_0})\subset N(V)$, where we used (\ref{intft0}) 
above. As a consequence we have that $\ti\g_n([0,1])\subset N(V)$. The curve  $\g_n=(N|_V)^{-1}
\circ \ti\g_n:[0,1]\to V$ joins $x_n$ to $y_n$ and satisfies $L_\be(\g_n)=L_\nu(\ti\g_n)\to 0$.
In particular we have that $d_\be(x_n,y_n)\to 0$, hence $x=y$. Claim \ref{not5a} is
proved.

\begin{claim} \label{notF}$\wi V\subset M$.
\end{claim}

In fact, if this is not true,  there exists $q\in\de M\cap\wi V$. Since $q\notin V$, Claim \ref{not5a} implies that $\wi N(q)\notin N(V)$. Thus (\ref{intft0}) above implies  that
$$\wi N(q)\in \lf(\hat F_{t_0}-N(V)\rg)\subset \p\hat F_{t_0}.$$

Since $\wi N(q)\notin N(V)$, the point $\wi N(q)$ does not satisfy condition (\ref{partition})-(a) in
Lemma \ref{general_lemma}. By condition (\ref{partition})-(b) in Lemma \ref{general_lemma} there exist $p\in S^2$, $\fr \pi 2<r<\pi$
and $\ep>0$ such that
$$\lf(B_\nu\bl(\wi N(q),\ep\br)\,\cap\, \inte\bl(\hat F_{t_0}\br)\rg)\subset
B_\nu(p,r).$$
Again by condition (\ref{partition})-(b) in Lemma \ref{general_lemma} the boundaries of $\bar B_\nu(p,r)$ and $\hat F_{t_0}$ coincide in a neighborhood
of $\wi N(q)$. More precisely, there exists a nonempty open arc $C$ satisfying
$$\wi N(q)\in C\subset
\lf(\p \hat F_{t_0}\cap S_\nu(p,r)\rg).$$
Thus we obtain that for some $\ep'>0$
sufficiently small it holds that
$$\Omega=\lf(B\bl(\wi N(q),\ep'\br)\,\cap\, \inte\bl(\hat F_{t_0}\br)\rg)=
\lf(B\bl(\wi N(q),\ep'\br)\,\cap\, B_\nu(p,r)\rg).$$

By (\ref{intft0}) above we have that $\Omega\subset \inte\bl(\hat F_{t_0}\br)\subset N(V)$.
We consider the open set $U=(N|_V)^{-1}(\Omega)$. Since $\wi U\subset \wi V$ we obtain
from Claim \ref{not5a} that $\wi N$ is injective on $\wi U$. Consider a sequence 
$q_n\in V$ converging to $q$. For sufficiently large $n$ we have that $N(q_n)\in \Omega$, 
hence $q_n=(N|_V)^{-1}(N(q_n))\in U$, hence $q\in\wi U$. We
conclude that $\wi M$ is concave at $q$, which contradicts
our hypotheses. Claim \ref{notF} is proved.

\begin{claim} \label{surjwiV} $\wi N|_{\wi V}:\wi V\to \hat F_{t_0}$ is a bijection.
\end{claim}
Indeed, by (\ref{intft0}) above we have that $\inte\bl(\hat F_{t_0}\br)\subset N(V)$. 
Since Claim \ref{not5a} holds, in order to show Claim \ref{surjwiV}  it suffices to prove that any point in
$\p \hat F_{t_0}$ is in the image of $N|_{\wi V}$. By Condition (\ref{deformation}) 
in the statement of Lemma \ref{general_lemma},  given $z\in \p \hat F_{t_0}$ there exists 
a unit vector $v\in T_z(S^2)$ pointing to
$\inte(\hat F_{t_0})$. We consider the unit speed $\nu$-geodesic $\ti\g:[0,\eta]\to \hat F_{t_0}$
given by $\ti\g(t)=\exp_z(\eta-t)v$, for some small $\eta>0$ such that $\ti\g([0,\eta))
\subset \inte\bl(\hat F_{t_0}\br)$. Set $\g:[0,\eta)\to
V$ given by $\g=(N|_V)^{-1}\circ\bl(\ti\g|_{[0,\eta)}\br)$. Take a
sequence $t_n\in [0,\eta)$ converging to $\eta$. Since $\g(t_n)$ is a Cauchy
sequence in $M$, it converges to some $q\in \wi V$ and $\wi N(q)=\lim N(\g(t_n))=
\lim\ti\g(t_n)=z$. Claim \ref{surjwiV}
is proved.

\begin{claim} \label{not_bound}$\wi V\subset M-\p M$.
\end{claim}
 By Claim \ref{notF}, there is nothing to prove if $\p M=\emptyset$. 
 Thus we will assume that $\p M\not=\emptyset$. 
 Fix $q\in \wi V$. We know by Claim \ref{notF} that 
 $q\in M$. If $q\in F$, we have by hypothesis that $q\notin\p M$. If 
$q\notin F$, Claim \ref{not5a} implies that $N(q)\notin N(F)=\hat F_{s_0}$. 
Thus condition (\ref{deformation}) in Lemma \ref{general_lemma} implies 
that there exists some $s_0< t\le t_0$ such 
that $N(q)\in \p\hat F_t$, hence $N(q)\in \lf(\p\hat F_t\cap \bl(S^2-N(F)\br)\rg)$. Since $N:\wi V\to \hat F_{t_0}$ is a homeomorphism which preserves the length of curves,
it follows from condition (\ref{star_shaped}) in Lemma \ref{general_lemma} that 
$$d_\be(w,\p M)>L_\nu(\g_{_{N(q)t}})=L_\be\lf(\bl(N|_{\wi V}\br)^{-1}\circ \g_{_{N(q)t}}\rg)\ge d_\be(w,q).$$ 
Thus we have that
$$d_\be(q,\p M)\ge d_\be(w,\p M)-d_\be(w,q)>0,$$
which implies that $q\in  M-\p M$. Claim \ref{not_bound} is proved.

Now we may get a contradiction and prove Lemma \ref{general_lemma}. Set $F_{t_0}=\wi V$. We have that $F\subset F_u\subset F_v\subset M-\p M$ for
$s_0\le u\le v\le t_0$. By Claim \ref{surjwiV} the map 
$N|_{F_{t_0}}:F_{t_0}\to \hat F_{t_0}$ is a homeomorphism, hence the set $F_{t_0}$ is
compact. We conclude that $t_0\in Z$, which contradicts Claim \ref{t0Z} and proves Lemma \ref{general_lemma}.
\end{proof}

\begin{lemma} \label{geodesicpi} Consider a connected noncompact surface  $M$ and a 
$C^1$ local isometry $N:(M,\be)\to (S^2,\nu)$ such that there exists no point $q\in\de M$ 
at which $\wi M$ is concave. Fix $w\in M-\p M$. If $\p M\not=\emptyset$ assume further that $d_\be(w,\p M)>\pi$. 
Then does not exist
any unit speed $\be$-geodesic $\g:\lf[0,\pi\rg]\to M$ with $\g\lf(\fr\pi 2\rg)=w$. 
\end{lemma}

\begin{proof}
Assume  by contradiction that such a geodesic $\g$ exists. Since $N:(M,\be)\to S^2$ is a 
$C^1$ local
 isometry we have that $N|_{\g([0,\pi])}$ is injective. By
 a compactness argument there exists a small 
 $0<\de<\fr\pi 2$ such that the set $F=\bar B_\be\bl(\g([0,\pi]),\de\br)$ 
 is a domain with $C^1$ boundary contained in $M-\p M$ and $N$ is injective 
 on an open set containing $F$, hence $N(F)$ is a domain in $S^2$ with $C^1$ boundary.  
 The number $\de$ may be chosen sufficiently small such that $N(F)=
 B_\nu\bl(\g([0,\pi]),\de\br)$. 
  
Now we rotate $\ti \g=(N\circ\g)$ in both directions fixing $\ti \g(0)$ and $\ti \g(\pi)$, obtaining
unit speed $\nu$-geodesics $\ti\g_s:[0,\pi]\to S^2$ such that $\ti \g_0=\ti\g$ and the angle $\measuredangle(\ti\g_s'(0),\ti\g_{-s}'(0))=2s$.

For $0\le t\le \fr\pi 2$, set
$$\hat F_t=\bigcup_{|s|\le t}\bar B_\nu(\ti\g_s([0,\pi]),\de)=
\bar B_\nu\lf(\Biggl(\bigcup_{|s|\le t}\ti\g_s([0,\pi])\Biggr),\de\rg).$$

Note that $\hat F_{\fr\pi 2} = \bar B_\nu\bl(N(w),\fr\pi 2+\de\br)$.
For $\fr \pi 2\le t<\pi-\de$ set $\hat F_t=\bar B_\nu\bigl(N(w),t+\de\bigr)$.

We claim that the collection of compact domains $(\hat F_t)_{0\le t<\pi-\de}$ satisfies the conditions
of Lemma \ref{general_lemma}. Note that $\hat F_0=\bar B_\nu\bl(\ti g\bl([0,\pi]),\de\br)\br)=N(F)$,
hence condition (\ref{F0}) in Lemma \ref{general_lemma} is satisfied. Conditions (\ref{monotone}) and (\ref{deformation}) in Lemma \ref{general_lemma} are also easily verified.

Now we will check that condition (\ref{partition}) in Lemma \ref{general_lemma} is
verified. If $0< t<\fr\pi 2$, the boundary $\p \hat F_t$ is a union of an arc in $S_\nu(\ti\g(0),\de)\subset N(F)$,
an arc in $S_\nu(\ti\g(\pi),\de)\subset N(F)$, whose points verify condition (\ref{partition})-(a) in
Lemma \ref{general_lemma}, an arc $C_{t}$ contained in $S_\nu\lf(p_{t},\fr \pi 2+\de\rg)$
for some $p_{t}\in S^2$,
such that $\inte(\hat F_t)$ and $B_\nu\lf(p_{t},\fr \pi 2+\de\rg)$ are locally on the same side
of $C_{t}$, and    an arc $C_{-t}$ contained in $S_\nu\lf(p_{-t},\fr \pi 2+\de\rg)$ for
some $p_{-t}\in S^2$
such that $\inte(\hat F_t)$ and $B_\nu\lf(p_{-t},\fr \pi 2+\de\rg)$ are locally on the same side
of $C_{-t}$. The points in $C_{t}$ and $C_{-t}$ verify condition (\ref{partition})-(b). More
precisely, let $E_t$ denote the equator containing the image of $\ti \g_t$ and
$H_{t}$ one of the two hemispheres determined by $E_t$ such that $\ti\g_{-t}\bl((0,\pi)\br)\subset H_{t}$. Take $p_{t}\in H_{t}$ such that $d_\nu(p_{t},E_t)=\fr \pi 2$.
 Similarly we obtain a point $p_{-t}$ by using the equator given by the geodesic $\ti \g_{-t}$.   
 If $\fr \pi 2\le t<\pi-\de$
 we have that $\p\hat F_t=S_\nu(N(w),t+\de)$ and thus condition (\ref{partition})-(b)
 is automatically verified.

 If $\p M\not=\emptyset$, let us check condition (\ref{star_shaped}) in Lemma \ref{general_lemma}. 
 We fix $0< t<\pi-
 \de$ and $z\in \p\hat F_t$ such that $z\notin N(F)$. It suffices to show that there exists a piecewise smooth curve $\g_{zt}:[0,1]\to
 \hat F_t$ from $N(w)=\ti\g(0)$ to $z$ with
 $L_\nu(\g_{zt})<\pi$. Assume first that $0<t<\fr\pi 2$. Since $\measuredangle\bl(\ti\g_t'(0),\ti\g_{-t}'(0)\br)=2t<\pi$, the 
 compact domain $D_t$ bounded by $\ti\g_t([0,\pi])$ and $\ti\g_{-t}([0,\pi])$ which contains  $\ti\g([0,\pi])$ is convex. Note that $\p \hat F_t=S_\nu(D_t,\de)$. Let $P(z)$ be the natural projection of $z$ onto $\p D_t$. Then we may construct a piecewise smooth curve $\g_{zt}$ joining $N(w)$ 
 to $z$, which first follows a minimal geodesic from $N(w)$ to $P(z)$, whose $\nu$-length 
is not greater then $\fr\pi 2$, by an argument 
similar as in the proof of Lemma \ref{Ls}, and then follows 
a minimal geodesic from $P(z)$ to $z$. We will have that $L_\nu(\g_{zt})\le\fr\pi 2+\de<\pi$. Finally we consider $\fr\pi 2\le t<\pi-\de$. Since $\p\hat F_t=S_\nu(N(w),t+\de)$,
 we define $\g_{zt}$ as a minimal geodesic from $N(w)$ to $z$, which satisfies 
 $L_\nu(\g_{zt})= t+\de<\pi$. We conclude that condition (\ref{star_shaped}) in 
 Lemma \ref{general_lemma} is satisfied.
  
 Thus we may use Lemma \ref{general_lemma} to obtain the existence of 
 a connected set $U\subset M-\p M$ containing $F$  
 such that $N|_U:U\to \bigcup_{0\le t<\pi-\de}\hat F_t=B_\nu\bl(N(w),\pi\br)=S^2-\{N(w)^*\}$ is a bijection.  
  By using the fact that 
$N$ is a $C^1$ local isometry we see that $N|_U$ is a homeomorphism which 
preserves lengths. Thus, by considering the isometry $dN_w:T_wM\to T_{N(w)}M$, we 
see that each unit speed $\be$-geodesic starting at $w$ will be defined at least on 
$[0,\pi)$. This implies that $B_\be(w,\pi)$ is a normal ball contained in $U$. Given a point $p\in U$, there exists a minimal $\nu$-geodesic 
$\si$ from $N(w)$ to $N(p)$ with $L_\nu(\si)<\pi$. Set $\g=(N|_U)^{-1}\,\circ\,\si$. 
We have that $L_\be(\g)=L_\nu(\si)<\pi$. Since $\g$ joins $w$ and $p$ we 
obtain that $d_\be(w,p)<\pi$, hence $U= B_\be(w,\pi)$. Consider divergente sequences 
$p_n,q_n$  in $U$. Since $N(p_n)$ and $N(q_n)$ converge to $(N(w))^*$, there 
exists a curve $\ti\ta_n$ joining $N(p_n)$ and $N(q_n)$ in $S^2-{N(w)^*}$ 
with $L_\nu(\ti\ta_n)\to 0$, hence $L_\be\bl((N|U)^{-1}\circ\ta_n\br)\to 0$. Thus 
we have that $d_\be(p_n,q_n)\to 0$. As a consequence we have that $\wi U-U=\{q\}$ 
and $\wi N(q)=(N(w))^*$.

We first consider the case that $q\in M$. Since $d_\be(w,q)=\pi$ we have 
that $q\notin \p M$, hence $\wi U$ is a bounded surface 
without boundary, hence it is compact and agrees with $M$, which contradicts the fact that $(M,\be)$ is not 
compact. 

The second possibility
is that $q\in \wi U\cap\de M$. 
Since $N|_U:U\to B_\nu\bl(N(w),\pi\br)$ is a bijection and $\wi U=U\cup\{q\}$, we conclude that $\wi N:\wi U\to S^2$ is a bijection. Take
$y\in S^2$ with $0<d_\nu\bl(N(w),y\br)=r<\fr \pi 2$. Then $\wi N(q)\in S_\nu(y,\pi-r)$. 
Set $B=N^{-1}\bl(B_\nu(y,\pi-r)\br)$. Note
that $\pi-r>\fr\pi 2$ and that $N|_{\wi B}:\wi B\to S^2$ is injective, hence
$\wi M$ is concave in $q$, which contradicts our hypotheses. The proof of Lemma \ref{geodesicpi} is complete.
\end{proof}

\begin{lemma} \label{small_delta} Consider a connected noncompact surface  $M$ and a 
$C^1$ local isometry $N:(M,\be)\to (S^2,\nu)$ such that there exists no point $q\in\de M$ 
at which $\wi M$ is concave. Fix $p\in M-\p M$ and $\de_0>0$. If $\p M\not=\emptyset$, assume further that $0<\de_0\le\fr{d_\be(p,\p M)}5$. Then for any distinct points $z,u\in \bar B=\bar B_\be(p,\de_0)$, there exists a unique unit speed minimal geodesic 
$\g:[0,d]\to M-\p M$ from $z$ to $u$ with  $L_\be(\g)<\pi$. In particular $N(z)\not=N(u)$, hence the map $N|_{\bar B}:\bar B\to S^2$ is injective.
 \end{lemma}
\begin{proof} By the triangle inequality we have that $\mbox{diam}_\be(\bar B)\le 2\de_0$.
Fix distinct points $z,u\in \bar B$. Since any metric ball in a Riemannian manifold is
path-connected,
 there exists a continuous curve $\si:[0,a]\to \bar B$ from $z$ to $u$.
By compactness there exist 
$0<r<\min\lf\{\fr{\de_0}2,\pi\rg\}$ and a partition
$0=t_0<t_1<\cdots<t_k=a$ of the interval $[0,a]$
such that for all $0\le i\le k-1$ it holds that $\bar B_i=\bar B_\be\bl(\si(t_i),r\br)$ is
a strongly convex normal ball in $M-\p M$ whose interior $B_i$ contains $\si(t_{i+1})$. 
Without loss of generality we may assume that $\si(t)\not= z$ for $0<t\le a$.

We will show by induction that for any $1\le i\le k$, there exists a unique unit speed minimal $\be$-geodesic
$\g_i:[0,s_i]\to M-\p M$ from $z=\si(0)$ to $u_i=\si(t_i)$ with $L_\nu(\g_i)<\pi$.   
 For $i=1$ this assertion is trivial. Assume that for some $1\le i
 \le k-1$ there exists a unique unit speed minimal $\be$-geodesic $\g_i:[0,s_i]\to M-\p M$ from
 $z=\si(0)$ to $u_i=\si(t_i)$. Set $w_i=\g_i\lf(\fr {s_i} 2\rg)$. Since $\mbox{diam}_\be(\bar B)\le 2\de_0$ we have that 
 \begin{equation}\label{siover2} d_\be(z,w_i)=
 \fr {d_\be(z,u_i)}2=\fr{s_i}2\le\de_0.
 \end{equation}
 
 If $\p M\not=\emptyset$ we obtain from (\ref{siover2}) that
 \begin{equation}\label{from_boundary}d_\be(w_i,\p M)\ge d_\be(p,\p M)-d_\be(p,z)-d_\be(z,w_i)\ge d_\be(p,\p M)-2\de_0\ge 3\de_0.\end{equation}
 
 Since $\bar B_0$ and $\bar B_i$ are normal balls with the same radius $r$, we may extend $\g_i$ in both directions  obtaining a unit speed $\be$-geodesic $\xi_i:[-r,s_i+r]\to M$. 
 
 \begin{claim}  \label{less_pi}$s_i+2r<\pi$, hence $s_i<\pi$ and $r<\fr \pi 2$.
 \end{claim}
 In fact, if $\p M=\emptyset$, this inequality follows immediately from 
 Lemma \ref{geodesicpi} applied to the geodesic $\xi_i$. If $\p M\not=\emptyset$  
 we assume by contradiction that $s_i+2r\ge \pi$. Since $r$ was chosen so that $r<\fr{\de_0}2$, we obtain from (\ref{siover2}) and (\ref{from_boundary}) that $d_\be(w_i,\p M)\ge 3\de_0\ge s_i+\de_0>s_i+2r\ge\pi$, hence $d_\be(w_i,\p M)>\pi$. By Lemma \ref{geodesicpi} we conclude that $s_i+2r< \pi$, and this contradiction proves Claim \ref{less_pi}. 
   
 Consider the set
$$\Omega=\Omega_i=\bar B_0\cup \g_i([0,s_i])\cup \bar B_i.$$

\begin{claim}\label{NinjOmega} $N$ is injective on $\Omega$.
\end{claim} 
In fact, by Claim \ref{less_pi} we have that $r<\fr\pi 2$ and $s_i<\pi$. 
Thus the fact that $N$ is a $C^1$ local isometry implies that the maps $N|_{\bar B_0}$, $N|_{\bar B_i}$ and $N|_{\g_i([0,s_i])}$ are injective. 
The map $N$ is injective on $\bar B_0\cup\g_i([0,s_i])$ and on $\bar B_i\cup\g_i([0,s_i])$), since $r<\fr \pi 2$ and then $N(\bar B_0)$ and  $N(\bar B_i)$ are strongly convex. Thus, to 
show that $N$ is injective on $\Omega$ it suffices to show that $N$ is injective on $\bar B_0\cup\bar B_i$. Fix $x\in \bar B_0$ and $y\in \bar B_i$ such that $v=N(x)=N(y)$. Since 
$N$ is a $C^1$ local isometry we 
have that 
$$2r\ge d_\be(z,x)+d_\be(u_i,y)\ge d_\nu(N(z),v)+d_\nu(N(u_i),v)\ge 
d_\nu(N(z),N(u_i))=s_i.$$ 
The fact that $2r\ge s_i$ implies that $w_i\in 
\bar B_0\cap\bar B_i$. Since $\bar B_0$ and $\bar B_i$ are strongly convex 
balls, there exist  unit speed minimal $\be$-geodesics $\la_1:[0,a_1]\to \bar B_0$ from 
$w_i$ to $x$ and $\la_2:[0,a_2]\to \bar B_i$ from $w_i$ to $y$. We have that 
$L_\nu(N\circ\la_j)\le 2r<\pi$ for $j=1,2$, and $(N\circ\la_j)(0)=N(w_i), (N\circ\la_j)(a_j)=v$, for $j=1,2$. Thus 
we have that $a_1=a_2$ and $N\circ\la_1=N\circ\la_2$. Since $N$ is a $C^1$ local isometry we obtain 
that $\la_1'(0)=\la_2'(0)$, hence $x=y$.  Claim \ref{NinjOmega} is proved.

By compactness there
exists $0<\ep<r$ sufficiently small such that for any $t\in [0,s_i]$ the ball 
$\bar B_\be(\g_i(t),\ep)$ is a normal ball in $M-\p M$ and $N$ is injective on 
$$F=F_\ep=\bar B_0\cup\bigcup_{t\in[0,s_i]}\bar B_\be\bl(\g(t),\ep\br)\,\cup\, \bar B_i=\bar B_0\cup
\bar B_\be\bl(\g_i\bl([0,s_i]\br),\ep\br)\cup \bar B_i.$$
Set:
$$\ti \g_i =N\circ \g_i, \hat B_0=N(\bar B_0)=\bar B_\nu\bl(N(z),r\br), \hat B_i=N(\bar B_i) = 
\bar B_\nu\bl(N(u_i),r\br), 
X=N(\Omega).
$$ 
In the particular case that $2r>s_i$, we may use the fact that $N$ is a $C^1$ local isometry, which is injective on the compact domain 
$\Omega=\bar B_0\cup\g_i\bl([0,s_i]\br)\cup \bar B_i=\bar B_0\cup \bar B_i$,  
to obtain that the number $\ep>0$ 
as above may be chosen not so small  such that $\hat B_0\cup \hat B_i$ is properly 
contained in $\hat B_0\,\cup\,\bar B_\nu\bl(\ti\g_i([0,s_i]),\ep\br)\,\cup\, \hat B_i$. 
 More precisely, if $2r>s_i$ we set $r_0=d_\be(w_i,\p \bar B_0\cap\p \bar B_i)>0$. 
Since $N$ is injective and a $C^1$ local isometry on $\bar B_0\cup \bar B_i$, each $\bar B_\be(\g_i(t),r_0)$ is a normal ball contained in $\bar B_0\cup \bar B_i$, for $0\le t\le s_i$, 
and $N$ is injective on $\bar B_\be\bl(\g_i([0,s_i]),r_0\br)$. Thus by compactness there 
exists $r_0<\ep<r$ such that for any $t\in [0,s_i]$ the ball 
$\bar B_\be(\g_i(t),\ep)$ is a normal ball in $M-\p M$ and $N$ is injective on $F_\ep=\hat B_0\cup\bar B_\nu\bl(\ti\g_i\bl([0,s_i),\ep\br)\br)\cup \hat B_i$.

The idea now is to construct the convex hull of $X=N(\Omega)$ as the union $\hat F_T$ 
of an increasing family of sets $\bl(\hat F_t\br)_{\ep\le t\le T}$ which satisfies the conditions in Lemma \ref{general_lemma} for $0\le t<T$.  Set $\hat F_{\ep}=N(F)$. For $\ep\le t\le r$, set $\hat F_t
=\hat B_0\cup\bar B_\nu\bl(\ti\g_i\bl([0,s_i]\br),t\br)\cup \hat B_i$. Note that $\hat F_r=\bar B_\nu\bl(\ti\g_i\bl([0,s_i]\br),r\br)$.

To define $\hat F_{t}$ for $t>r$, we write $t=s+r$ for convenience. 
Consider a unit speed $\nu$-geodesic $\eta:\Bbb R\to S^2$ orthogonal to $\ti\g_i$ at 
$\eta(0)=N(w_i)=\ti\g_i\lf(\fr{s_i}2\rg)$. For $0\le s\le\fr\pi 2$, set $q_{s}=
\eta\lf(-\fr\pi 2+s\rg)$ and $q^s
=\eta\lf(\fr\pi 2-s\rg)$. By symmetry we have that
$$d_s=d_\nu\bl(q_{s},N(z)\br)=d_\nu\bl(q_{s},N(u_i)\br) =d_\nu\bl(q^{s},N(z)\br)=d_\nu\bl(q^{s},N(u_i)\br).$$ 

By Claim \ref{less_pi} we have that $\fr {s_i}2<\fr\pi 2$. Thus equation (\ref{cosines}) implies that,  
if $0\le s\le \fr\pi 2$, it holds that $d_s\le\fr\pi 2$, hence by Claim \ref{less_pi} 
it holds that $d_s+r\le \fr\pi 2+r<\pi$. As a consequence there exists a unique unit speed minimal $\nu$-geodesic $\va_{s}:[0,d_{s}+r]\to S^2$ satisfying 
$\va_{s}(0)=q_{s}$ and $\va_s(d_s)=N(z)$.  Set $\psi_s=R\circ\va_s:[0,d_{s}+r]\to S^2$, where $R$ is the reflection which fixes the image of $\eta$. Set $\va^s=S\circ\va_s$ and $\psi^s=S\circ\psi_s$, 
where $S$ is the reflection which fixes the equator containing the image of $\ti\g_i$.  

By (\ref{cosines}) we have that 
$d_{{0}}=\fr\pi 2$ and that the map $s\in\lf[0,\fr\pi 2\rg]\longmapsto d_s$ is strictly decreasing, hence the number $d_s+r$ decreases 
from $\fr\pi 2+r$ to $\fr {s_i} 2+r$.   By Claim \ref{less_pi} we have that $\fr {s_i}2+r<\fr\pi 2$, hence there exists a unique 
$0<\bar s<\fr\pi 2$ such that $d_{\bar s}+r=\fr\pi 2$. Since $\va_s'(d_s+r)$ is 
orthogonal to both $S_\nu\bl(N(z),r\br)$ and $S_\nu\bl(q_s,d_s+r\br)$, we obtain 
that $S_\nu(q_s,d_s+r)$ is tangent to 
both $\p\hat B_0=S_\nu\bl(N(z),r\br)$ and $\p\hat B_i=S_\nu\bl(N(u_i),r\br)$. Similarly we 
obtain that $S_\nu(q^s,d_s+r)$ is tangent to $\p \hat B_0$ and $\p\hat B_i$. 

For $0\le s\le \bar s$, consider a $C^1$ piecewise smooth 
simple closed curve $\ta_s$, which follows $S_\nu(q_s,d_s+r)$ 
from $\va_s(d_s+r)$ to $\psi_s(d_s+r)$, then $\p \hat B_i$ from $\psi_s(d_s+r)$ to 
$\psi^s(d_s+r)$, then $S_\nu(q^s,d_s+r)$ from $\psi^s(d_s+r)$ to $\va^s(d_s+r)$, 
and then $\p \hat B_0$ from $\va^s(d_s+r)$  to $\va_s(d_s+r)$. For 
$0<s\le \bar s$, let $\hat F_{s+r}$ 
be the domain which contains $X$ and is bounded by the image of $\ta_s$. Note 
that $\hat F_{\bar s}+r$ agrees with the convex hull of $X$.  Set $T=\bar s+r$. 

Now we will prove that the family $(\hat F_t)_{\ep\leq t<T}$ satisfies the conditions 
in Lemma \ref{general_lemma}. We will see that Conditions (\ref{monotone}) and 
(\ref{deformation}) in Lemma \ref{general_lemma} hold even for $\ep\le t\le T$. Conditions (\ref{F0}) and (\ref{partition}) in Lemma 
\ref{general_lemma} are 
trivially satisfied. 

\begin{claim} \label{pr_deformation} Condition {\rm(\ref{deformation})} in Lemma 
{\rm\ref{general_lemma}} holds for $\ep\le t\le T$.
\end{claim}
From the facts that $d_0=\fr \pi 2$ and that the image of $\ti\g_i$ is contained in 
 $S_\nu\lf(q_{_0},\fr\pi 2\rg)=S_\nu\lf(q^0,\fr\pi 2\rg)$, we 
obtain easily that the image of $\ta_{_0}$ agrees with $\p \hat F_r=S_\nu\bl(\ti\g_i([0,s_i],r)\br)$. This shows the continuous dependence of $\p\hat F_t$ on the parameter $\ep\le t\le T$ 
(note that the homotopy $s\longmapsto \ta_s$ extends easily to $s=\bar s$). 
Thus to see that condition (\ref{deformation}) in Lemma \ref{general_lemma} 
holds it suffices to verify that each $\p\hat F_t$ is the image of a piecewise smooth 
simple closed curve with internal angles different from $0$ and $2\pi$. This is clear for 
$r\le t\le T$, since each $\ta_s$ is $C^1$ and piecewise smooth. Fix $\ep\le t<r$. 
The boundary $\p \hat F_t$ has $4$ vertices. Let $x_1$ be the vertex contained 
in $S_\nu(\ti\g_i(0),r)\cap S_\nu\lf(q_{_0},\fr\pi 2+t\rg)$.  
 Since  
the intersection between the circles $S_\nu(\ti\g_i(0),r)$ and 
$S_\nu\lf(q_{_0},\fr\pi 2+t\rg)$ contains exactly two points for $\ep\le t<r$, 
then they intersect 
themselves transversely, relatively to the ambient space $S^2$. In particular the 
corresponding internal angle at the vertex $x_1$ 
is different from $0$ and $2\pi$. By symmetry we conclude the same fact 
about the other $3$ vertices. Thus condition (\ref{deformation}) in Lemma \ref{general_lemma} 
holds.

\begin{claim} \label{pr_monotone} Condition {\rm(\ref{monotone})} in Lemma 
{\rm\ref{general_lemma}} holds for $\ep\le t\le T$. 
\end{claim}
If $\ep\le t\le t'\le r$,  we have by construction that $\hat F_t\subset \hat F_{t'}$. For 
$r\le t< t'\le T$ write $t=r+s$ and $t'=r+s'$. Set $D_0=\ti\g_i\bl([0,{s_i}]\br)$ 
and 
$$D_s=\bar B_\nu(q_s,d_s)\cap \bar B_\nu(q^s,d_s),$$
if $0<s\le \bar s$. Thus Claim \ref{pr_monotone} follows from Lemma \ref{Ls}. 

If $\p M\not=\emptyset$ we will verify that condition (\ref{star_shaped}) in Lemma \ref{general_lemma} 
holds. We will show that for any $\ep< t< T$ and any $x\in \p\hat F_t$ with 
$x\notin N(F)$, there 
exists a piecewise smooth curve $\g_{xt}:[0,1]\to \hat F_t$ joining $N(w_i)$ and $x$ such 
that $L_\nu(\g_{xt})<d_\be(w_i,\p M)$.  First assume that $\ep<t\le r$. In this 
case we have 
that $x\in S_\nu\lf(q_0,\fr \pi 2+t\rg)\cup S_\nu\lf(q^0,\fr \pi 2+t\rg)$. Let 
$P(x)$ be the natural projection of $x$ onto $\ti\g_i([0,s_i])$. Let $\g_{xt}$ be 
the piecewise smooth curve which follows the image of $\ti\g_i$ from 
$N(w_i)$ to $P(x)$ and then follows the minimal geodesic from $P(x)$ to $x$. 
Since $B_\nu(\ti\g_i
([0,s_i]),t)\subset \hat F_t$, we have that the image of $\g_{xt}$ is contained in 
$\hat F_t$. By using inequality (\ref{from_boundary}) above, we have that 
$$L_\nu(\g_{xt})\le \fr {s_i} 2+t\le\fr {s_i}2+r<\de_0+\fr{\de_0}2=\fr{3\de_0}2<3
\de_0\le d_\be(w_i,\p M).$$
Now assume that $t=r+s$ for some $0<s<s_1$. Thus $x\in S_\nu(D_s,r)$. Let $P_1(x)$ 
be the natural projection from $x$ to $\p D_s$. We define some piecewise smooth 
curve $\g_{xt}:[0,1]\to \hat F_t$, which follows a minimal geodesic $\chi_{_1}$ 
from $N(w_i)$ to 
$P_{_1}(x)$ and then a minimal geodesic from $P_{_1}(x)$ to $x$.  Since $D_s$ is strongly 
convex we have that the image of $\g_{xt}$ is contained in $\hat F_t$. We 
claim that $L_\nu(\chi_{_1})\le \fr{s_i}2$. Indeed, similarly 
as in the proof of Lemma \ref{Ls} we see that the map $y\in \p D_s\longmapsto 
d_\nu(N(w_i),y)$ attains its maximum at $N(z)$ and $N(u_i)$,  hence we have that  
$$L_\nu(\g_{xt})=L_\nu(\chi_{_1})+r\le d_\nu(N(w_i),N(z))+r=\fr {s_i}2+r< \fr{3\de_0}2<d_\be(w_i,\p M).$$ 

Thus we may use Lemma \ref{general_lemma} to obtain that there exists a connected 
set $U_i$ with $F\subset U_i\subset M-\p M$ such that $N|_{U_i}:U_i\to \bigcup_{\ep\le t<T}\hat F_t$ is a bijection and a $C^1$ local 
isometry. From the continuity of the map $r\le s\le s_1\longmapsto \ta_s$ we see that 
$\inte(\hat F_T)\subset\bigcup_{\ep\le t<T}\hat F_t$. 
The set $\hat F_T$ is strongly convex, hence we have that $\inte(\hat F_T)$ is also 
strongly convex, since the map $x\in \hat F_T\to d_\nu\bl(x,\p\hat F_T\br)$ is concave (see Theorem 1.10 
in \cite{cg}). Since $N|_{U_i}$ is a bijection which is a $C^1$ local isometry we obtain 
easily that $W_i=(N|_{U_i})^{-1}\bl(\inte(\hat F_T)\br)$ is strongly convex. Since $\si(t_{i+1})\in 
B_i=B_\be(u_i,r)$ we have that $z$ and $\si(t_{i+1})$ belong to $\inte(X)\subset W_i$. We conclude that there exists a unique 
unit speed minimal $\be$-geodesic $\g_{i+1}$ from $z$ to $u_{i+1}=\si(t_{i+1})$. 

We claim that the image of each $\g_{i+1}$ is contained in 
$M-\p M$. To show this we may assume that $\p M\not=\emptyset$.  
Given
$x$ in the image of $\g_{i+1}$, if $d_\be(z,x)\le d_\be(x,u_{i+1})$ we obtain 
an estimative as in inequality (\ref{from_boundary}), obtaining that $d_\be(x,\p M)\ge 
3\de_0$. If $d_\be(z,x)\ge d_\be(x,u_{i+1})$ we obtain similarly that $d_\be(x,\p M)\ge 
3\de_0$, hence $x\in M-\p M$. Thus we have that the image of $\g_{i+1}$ is contained in 
$M-\p M$. 

By triangle 
inequality we have that $L_\be(\g_{i+1})=L_\nu(N\circ\g_{i+1})\le s_{i}+r<\pi$, 
hence $L_\be(\g_{i+1})<\pi$. 
We conclude by induction that there exists a unique minimal $\nu$-geodesic $\g:[0,d]\to 
M-\p M$ 
from $z$ to $\si(t_k)=u$ with length less than $\pi$, whose image is 
contained in $M-\p M$. Since $N$ is a $C^1$ local isometry and $L_\be(\g)<\pi$, 
we have that 
$N$ is injective on the image of $\g$. In particular we have that $N(z)\not=N(u)$, hence $N$ is injective on $B_\be(p,\de_0)$. Lemma 
\ref{small_delta} is proved. 
\end{proof}

\begin{proof}[\bf Proof of the Efimov's Theorem] Assume by contradiction that 
there exists  an immersion 
$\varphi:M\to \Bbb R^3$
 as in the 
statement of the Efimov's Theorem. Let $\al$, respectively, $\be$ be the Riemannian
metrics induced by $\varphi$, respectively, by the Gauss map $N$.  
Lemmas \ref{concave} and \ref{small_delta} 
imply that $N$ is injective on $M$ and that any two points in $M$ are joined 
by a minimizing geodesic. In particular $M$ is simply-connected, hence 
$(M,\al)$ is a Hadamard surface. Thus we have that $A_\al(M)=+\infty$ 
 (see \cite{mi}). 
Since $|K|\ge \kappa>0$ we have that 
 $\int_M|K|dA_\al=+\infty$, hence by equation (\ref{infinite_area}) we 
have that $A_\be(M)=+\infty$. However, since $N$ is injective on $M$ and a $C^1$ local isometry we have that $A_\be(M)\le A_\nu(S^2)$, which gives us a contradiction and proves Efimov's Theorem.
\end{proof}

\begin{lemma} \label{Misbounded} Consider a connected noncompact surface  $M$ 
with compact boundary $\p M\not=\emptyset$ and a $C^1$ local isometry $N:(M,\be)\to (S^2,\nu)$ such that there exists no point $q\in\de M$ 
at which $\wi M$ is concave. Then $M$ is  bounded with respect to the metric $\be$. 
More precisely, $M\subset B_\be(\p M,5\pi)$. 
\end{lemma}
\begin{proof} Assume by contradiction that there exists a point $p\in M$ with 
$d_\be(p,\p M)\ge 5\pi$. Consider a continuous curve $\ta:[0,1]\to M$ with 
$\ta(0)=p$ and $\ta(1)\in \p M$. Thus there exists a point $q$ in the image of  
$\ta$ such that $d_\be(p,q)=\pi$. By Lemma \ref{small_delta} there exists a unit speed minimal 
$\be$-geodesic $\g:[0,\pi]\to M-\p M$ from $p$ to $q$. Set $w=\g\lf(\fr\pi 2\rg)$. 
We have that 
$$d_\be(w,\p M)\ge d_\be(p,\p M)-d_\be(p,w)=d_\be(p,\p M)-\fr\pi 2>\pi,$$
and this contradicts Lemma \ref{geodesicpi}. Lemma \ref{Misbounded} is proved.
\end{proof}

\begin{lemma}\label{Misunbounded} Consider a connected noncompact surface  $M$ 
with compact boundary $\p M\not=\emptyset$ and a $C^1$ local isometry $N:(M,\be)\to (S^2,\nu)$ such that there exists no point $q\in\de M$ 
at which $\wi M$ is concave. Then $(M,\be)$ is pre-compact.
\end{lemma}
\begin{proof} 
Since $\p M$ is a finite union of circles, there exists $0<\ep_0<\pi$
such that the set $C=\bar B_\be(\p M,\ep_0)$ is a closed collar neighborhood of $\p M$. 
Set $S_0= S_\be(\p M,\ep_0)$.

Consider a sequence $p_n$ in $M$. We need to prove that 
there exists a Cauchy subsequence of $p_n$. By compactness of $C$, 
if $d_\be(p_n,C)\to 0$ then $p_n$ has a convergent subsequence. 
Thus we may assume, by passing to a subsequence and using the fact that 
$M$ is bounded, that there exists 
some $d>0$ such that $d_n=d_\be(p_n,C)\to d$. For each $n$, there exists a unit speed piecewise smooth curve $\si_n:[0,L_n]\to M$ with 
$\si_n(0)\in S_0$ and $\si_n(L_n)=p_n$ 
such that 
$$L_n=L_\be(\si_n)<d+1\,.$$ 
By discarding a piece of the image of $\si_n$, if necessary, we may assume that $\si_n((0,L_n])\subset M-C$. Fix $k_0\in \Bbb N$ 
such that 
$$\fr{d+1}{k_0}<\fr{\ep_0}{10}.$$
Consider a partition $0=t_{n,0}<t_{n,1}<\cdots<t_{n,k_0}=L_n$, such that for 
any $0\le i\le k_0-1$ it holds that 
\begin{equation}\label{ep_ten}L_\be\bl(\si_n|_{[t_{n,\,i},\ t_{n,i+1}]}\br)=\fr {L_n}{k_0}\le \fr{d+1}{k_0}<\fr{\ep_0}{10}\,.
\end{equation}

Set $p_{n,i}=\si_n(t_{n,i})$. We will prove by induction that for each $0\le i\le k_0$ it holds 
that $p_{n,i}$ has a Cauchy 
subsequence, hence $p_n=p_{n,k_0}$ has a Cauchy subsequence. For $i=0$, the compactness of $S_0$ implies this easily.  
Assume that, for some $0\le i\le k_0-1$, we have, by passing to a subsequence, that $p_{n,i}$ is a Cauchy sequence in $(M,\be)$. We 
need to prove that $p_{n,\,i+1}$ has a Cauchy subsequence.  

There exists $n_1\in \Bbb N$ such that if $n\ge n_1$ then $d_\be(p_{n_1,i},\,p_{n,i})
<\fr{\ep_0}{10}$. Set $q_i=p_{n_1,i}$. By using (\ref{ep_ten}), we obtain that for any 
$n\ge n_1$ it holds that 
$$d_\be(q_i,\,p_{n,i+1})\le d_\be(q_i,\,p_{n,i})+
d_\be(p_{n,i},\,p_{n,i+1})<\fr{\ep_0}{10}+\fr{\ep_0}{10}=\fr{\ep_0}5.
$$

Thus for $n\ge n_1$ we have that $p_{n,i+1}\in B=B_\be\lf(q_i,\fr{\ep_0}5\rg)$. Lemma 
\ref{small_delta} implies that the map  $N|_B:B\to B_\nu\lf(N(q_i),\fr{\ep_0}5\rg)$ is injective 
and that, for any $m,n\ge n_1$, there exists a unique unit speed minimal 
$\be$-geodesic $\eta_{mn}$ from $p_{n,i+1}$ to $p_{m,i+1}$ with $L_\be(\eta_{mn})<\pi$. 
This implies that $N\circ\eta_{mn}$ is the unique speed minimal $\nu$-geodesic 
from $N(p_{n,i+1})$ to $N(p_{m,i+1})$. 

By passing to a subsequence, we may assume that $N(p_{n,i+1})$ is a Cauchy sequence 
in $(S^2,\nu)$. Given $\ep>0$, there exists $n_0\ge n_1$ such that if $m,n\ge n_0$ 
then $d_\nu\bl(N(p_{m,i+1}),N(p_{n,i+1})\br)=L_\nu(N\circ\eta_{mn})<\ep$. In particular 
we have that $d_\be(p_{m,i+1},\, p_{n,i+1})=L_\be(\eta_{mn})<\ep$, hence $p_{n,i+1}$ is a Cauchy sequence.  This completes the proof of Lemma \ref{Misunbounded}.
\end{proof}

\begin{lemma} \label{finite_area} Consider a connected noncompact surface  $M$ 
with compact boundary $\p M\not=\emptyset$ and a $C^1$ local isometry $N:(M,\be)\to (S^2,\nu)$ such that there exists no point $q\in\de M$ 
at which $\wi M$ is concave. Then $A_\be(M)$ is finite.
\end{lemma}
\begin{proof} As in the proof of Lemma \ref{Misunbounded}, there exists $0<\ep_0<\pi$
such that the set $C=\bar B_\be(\p M,\ep_0)$ is a closed collar neighborhood of $\p M$. 
Set $S_0= S_\be(\p M,\ep_0)$. 
Fix $0<\ep<\fr{\ep_0}5$. Since $M$ is pre-compact then $D=(M-C)\cup S_0$ is pre-compact. 
Thus there exist points $q_1,\cdots,q_{n_0}\in D$ such that 
\begin{equation} D\subset\bigcup_{1\le i\le n_0}
B_\be(q_i,\ep).
\end{equation}
Set $B_i=B_\be(q_i,\ep)$. By Lemma \ref{small_delta}, $N|_{B_i}$ is 
injective, hence 
\begin{eqnarray*}A_\be(D)&\le& \sum_{i=1}^{n_0}A_\be(B_i)=\sum_{i=1}^{n_0}A_\nu(N(B_i))
\le \sum_{i=1}^{n_0}A_\nu(B_\nu\bl(N(q_i),\ep)\br)\\
&=&n_0 A_\nu\bl(B_\nu(N(q_1),\ep)\br),
\end{eqnarray*}
hence $A_\be(M)\le A_\be(C)+n_0 A_\nu\bl(B_\nu(N(q_1),\ep)\br)$. Lemma \ref{finite_area}
is proved.
\end{proof}

\begin{proof} [\bf Proof of Theorem A] Assume by contradiction that there exists 
an immersion $\varphi:M\to \Bbb R^3$ as in the statement of 
Theorem A. Let $\al$, respectively, $\be$ be the Riemannian metrics 
induced by $\varphi$, respectively, $N$. By using Lemmas \ref{lemmaA} and \ref{finite_area} 
we obtain that  $A_\be(M)$ is finite, which contradicts equation (\ref{infinite_area}) and proves 
Theorem A.
\end{proof}

\section {\bf Appendix - Proof of Lemma \ref{topologies}}

To prove Lemma \ref{topologies} we first assume that $d_\inte$ is a distance on $D$. To show
that $d_g$ and $d_\inte$ induce the same topology on $D$, we need to prove that, given $q\in D$ and $\ep>0$, there exists $\de>0$ such that
if $p\in D$ with $d_g(p,q)<\de$ then $d_\inte(p,q)<\ep$. If $q\in \inte(D)$ the proof is
trivial. Thus we will assume that $q\in\p D$.

Fix $q\in\p D$ and $\ep>0$. For some small $\la>0$, there exists a curve $\si:[-\la,\la]\to \p D, \la>0$,  parameterized by the $g$-arc length satisfying that $\si(0)=q$ and such that
$\si|_{[-\la,0]}$ and $\si|_{[0,\la]}$ are smooth curves. Since $D$ is a piecewise smooth surface with boundary and the angles at the vertices differ from $0$ and $2\pi$, there exists a unit vector $v\in T_qS$ pointing to $\inte(D)$ and 
transversal to both $\si'(0-)$ and $\si'(0+)$. Let $v_t$ be the parallel
transport of $v$ along the both directions on $\si$. If $\la$ is small enough, we may assume that $v_t$ is
transversal to $\si'(t)$ and that $v_t$ points to $\inte(D)$. Set $\si_s(t)=\exp_{\si(t)}sv_t=\g_t(s)$. By smoothness of the geodesic flow there exists
sufficiently small $0<\eta<\min\lf\{\la,\fr \ep 3\rg\}$ such that:
\begin{enumerate}
\item for $0<s\le \eta$ and $-\eta\le t\le \eta$, the point $\si_s(t)\in \inte(D)$;
\item $L_g(\si_\eta)<\fr\ep 3$.
\end{enumerate}

Given $s_0\in [0,\eta]$ and $t_0\in[-\eta,\eta]$, we will construct a piecewise smooth
curve $\xi=\xi_{s_0t_0}:[0,2\eta+|t_0|-s_0]\to D$ from $q$ to $\si_{s_0}(t_0)$ satisfying $\xi((0,2\eta+|t_0|-s_0))\subset \inte (D)$ and $L_g(\xi)<\ep$.
From $q$ to $\si_{\eta}(0)=\g_0(\eta)$, let $\xi|_{[0,\eta]}$ coincide with the geodesic $\g_0:[0,\eta]\to D$. From $\si_{\eta}(0)$ to $\si_{\eta}(t_0)$ the curve $\xi$ follows the
curve $\si_\eta$ in the direction that $t$ is increasing if $0\le t_0$, or in
the other direction if $t_0<0$. More precisely, for $0\le s\le |t_0|$,
we define $\xi(\eta+s)=\si_\eta(s)$, if $0\le t_0$, and $\xi(\eta+s)=\si_\eta(-s)$,
if $t_0<0$. Finally, from $\si_{\eta}(t_0)$ to $\si_{s_0}(t_0)$ the curve $\xi$ follows the geodesic $s\longmapsto
\g_{t_0}(\eta-s)$. Namely, for
$0\le s\le \eta-s_0$ we
define $\xi(\eta+|t_0|+s)=\g_{t_0}(\eta-s)=\si_{(\eta-s)}(t_0)$. By construction we have that $L_g(\xi)\le\eta+
L_g(\si_{\eta})+(\eta-s_0)<\ep$. Given $s_1\in [0,\eta]$ and $t_1\in[-\eta,\eta]$, a similar 
construction as above shows that $\si_{s_0}(t_0)$ may be connected to $\si_{s_1}(t_1)$ by 
a piecewise smooth curve $\psi:[0,1]\to D$ with $\psi((0,1))\subset \inte(D)$ and $L_g(\psi)
<\ep$.

Set $X=\{\si_s(t)\bigm|0\le s\le \eta,\, -\eta\le t\le \eta\}$. Since $X$ is a compact neighborhood of $q$ in $D$, we have that $\de=d_g(q,D-X)>0$.
Now we take $p\in D$ with $d_g(p,q)<\de$. We have that $p\in X$, hence
$p=\si_{s_0}(t_0)$ for some $s_0\in [0,\eta]$ and $t_0\in [-\eta,\eta]$.  As a consequence
we have that $d_{\inte}(p,q)\le L_g(\xi_{s_0t_0})<\ep$. 

Now we consider points $p,q,r\in D$ and we will show that $d_\inte(p,q)+d_\inte(q,r)
\ge d_\inte(p,r)$. We will just consider the case that $q\in \p D$, since the other
is trivial. Fix $\ep>0$ and consider piecewise smooth curves $\g:[0,1]\to D$ from $p$ to $q$ with
$\g((0,1))\subset \inte(D)$ and $L_g(\g)<d_\inte(p,q)+\ep$, and
$\si:[0,1]\to D$ from $q$ to $r$ with
$\si((0,1))\subset \inte(D)$ and $L_g(\si)<d_\inte(q,r)+\ep$.
By using a neighborhood $X$ of $q$ as above, it is easy to obtain a piecewise smooth curve
$\varphi:[0,1]\to D$ from $p$ to $r$ with 
$\varphi((0,1))\subset \inte(D)$ and $L_g(\varphi)<L_g(\g)+L_g(\si)+\ep$. In fact, 
take $0<s_1<1$ such that $\g(s_1)\in X-\{q\}$ and $0<s_2<1$ such that $\si(s_2)\in X-\{q\}$. We define a curve $\varphi$ which 
follows $\g$ from $t=0$ to $t=s_1$ then
follows a curve $\psi$ in $X\cap\inte(D)$ with $L_g(\psi)<\ep$, and then
follows $\si$ from $t=s_2$ to $t=1$. Thus we obtain
that $d_\inte(p,q)+d_\inte(q,r)+3\ep> d_\inte(p,r)$. By making
$\ep\to 0$ we conclude the proof of
Lemma \ref{topologies}.
\begin{remark} It is not difficult to see that Lemma \ref{topologies} may be improved to assume that internal angles are just different from $0$, but 
this weaker assumption is not necessary for the proof of Theorem A. 
\end{remark}


\begin{thebibliography}{10}
\bibitem[Bi]{bi} Bieberbach, L.,  {\it Hilberts Satz \"uber Fl\"achen konstanter negativer Kr\"ummung}, Acta Math. {\bf 48} (1926), 319--327.
\bibitem[Bl]{bl} Blan\v usa, D., {\it\"Uber die Einbettung hyperbolischer R\"aume in euklidische R\"aume}, Monatsh. Math. {\bf 59} (1955), 217 -- 229. 
\bibitem[Bs]{bs} Blaschke, W., {\it Vorlesungen \"uber Differentialgeometrie}, p. 206, Springer, Berlin, 1924.
\bibitem[CG]{cg} Cheeger, J., Gromoll, D., {\it On the structure of complete manifolds of nonnegative curvature}, Ann. Math. {\bf 96}(3) (1972), 413--443.
\bibitem[C-V] {cv} Cohn-Vossen, S.,  {\it Bendability of surfaces in the large (Russian)}, Uspekhi Mat. Nauk,
{\bf 1} (1936), 33--76.
\bibitem[Ef]{e}  Efimov, V.,  {\it Generation of singularities on surfaces of negative curvature (Russian)},
Mat. Sbornik {\bf 64} (1964), 286--320.
\bibitem[GMT] {got} G\'alvez, J. A., Mart\'\i nez, A., Teruel, J. L., {\it Complete Surfaces with Ends of Non Positive Curvature}, http://arxiv.org/abs/1405.0851.
\bibitem[Hi]{h} Hilbert, D., {\it Ueber Fl\"achen von constanter Gausscher Kr\"ummung},
Math. Sot., {\bf 2} (1901), 87--99.
\bibitem[Ho]{ho} Holmgren, E., {\it Sur les surfaces \`a courbure constante negative}, C. R. Acad. Sci. Paris
{\bf 134} (1902), 740--743.
\bibitem[Ku]{ku} Kuiper, N. {\it On $C^1$-isometric imbeddings I, II}, Nederl. Akad. Wetensch. Proc. Ser. A
58; Indig. Math. 17 (1955), 545-556 and 683-689.
\bibitem [Mi]{mi} Milnor, T. K.,  {\it Efimov's theorem about complete immersed surfaces of negative curvature}, Advances
in Mathematics, {\bf 8} (1972), 474--543.
\bibitem [My]{my} Myers, S. B., {\it Riemannian manifolds with positive mean curvature}, Duke Math. J. {\bf 8}
(1941), 401--404.
\bibitem[Ro]{ro} Rozendorn, \`E. R., {\it A realization of the metric $ds^2=du^2+f^2(u)$ in a five-dimensional Euclidean space} (Russian. Armenian summary), Akad. Nauk Armjan. SSR Dokl. {\bf 30} (1960), 197--199.
\bibitem[SX]{sx} Smyth, B.,  Xavier, F., {\it Efimov's theorem in dimension greater than two}, Invent. Math. {\bf 90} (1987), 443--450.
\end{thebibliography}
\end{document}